\documentclass[a4paper]{article}

\usepackage{amsmath, amsthm, amsfonts}
\usepackage[utf8]{inputenc}
\usepackage{srcltx}
\usepackage{amssymb}
\usepackage[colorlinks=true,linkcolor=blue,citecolor=blue]{hyperref}

\usepackage{geometry}

\newcommand{\iO}{\int_{\Omega}}

\newcommand{\eps}{\varepsilon}

\def\N{\mathbb{N}}
\def\R{\mathbb{R}}

\def\pa{\partial}

\newtheorem{The}{Theorem}[section]
\newtheorem{cor}[The]{Corollary}
\newtheorem{Lem}[The]{Lemma}
\newtheorem{prp}[The]{Proposition}

\theoremstyle{definition}

\theoremstyle{remark}
\newtheorem{Rem}[The]{Remark}

\title{Global classical solutions for mass-conserving, (super)-quadratic reaction-diffusion systems 
in three and higher space dimensions}

\author{Klemens Fellner\footnote{Institut f\"ur Mathematik und
    Wissenschaftliches Rechnen, Heinrichstra{\ss}e 36, 8010
    Graz, Austria. Email: \texttt{klemens.fellner@uni-graz.at}. Partially supported by NAWI Graz.}
  \and Evangelos Latos\footnote{Institut f\"ur Mathematik und
    Wissenschaftliches Rechnen, Heinrichstra{\ss}e 36, 8010
    Graz, Austria. Email: \texttt{evangelos.latos@uni-graz.at}. Partially supported by NAWI Graz.}
  \and Takashi Suzuki\footnote{Graduate School of Engineering Science /
Department of Systems Innovation / Division of Mathematical Science, Osaka University. Email: \texttt{suzuki@sigmath.es.osaka-u.ac.jp}}
}

\def\R{I\!\!R}
\def\N{I\!\!N}
\def\pa{\partial}

\def\cd{\!\cdot\!}
\def\nx{\nabla_{\!x}}

\def\iO{{\int_{\Omega}}}

\begin{document}

\maketitle

\begin{abstract}
This paper considers quadratic and super-quadratic reaction-diffusion systems, for which all species satisfy uniform-in-time $L^1$ a-priori estimates, for instance, as a consequence of suitable mass conservation laws. A new result on the global existence of classical solutions is proved in three and higher space dimensions by combining regularity and interpolation arguments in Bochner spaces, a bootstrap scheme and a weak comparison argument. Moreover, provided that the considered system allows for entropy entropy-dissipation estimates proving exponential convergence to equilibrium, we are also able to prove that solutions are bounded uniformly-in-time. 
\\


\end{abstract}

\section{Introduction}

In this article, we first consider the following quadratic reaction-diffusion system:
\begin{equation}\label{sys4}
 \begin{cases}
 \pa_t u_1 - d_1\, \Delta_x u_1 = u_3\,u_4 - u_1\, u_2, &\qquad \text{on}\quad \Omega_T,\\
 \pa_t u_2 - d_2\, \Delta_x u_2 = u_3\,u_4 - u_1\, u_2, &\qquad \text{on}\quad \Omega_T,\\
 \pa_t u_3 - d_3\, \Delta_x u_3 = u_1\, u_2 - u_3\,u_4, &\qquad \text{on}\quad \Omega_T,\\
  \pa_t u_4 - d_4\, \Delta_x u_4 = u_1\, u_2 - u_3\,u_4, &\qquad \text{on}\quad \Omega_T,\\
 n(x)\cd\nx u_i(t,x) = 0, &\qquad \text{on}\quad \partial\Omega_T,\\
  u_i(0,x) = u_{i0}(x)\ge0, &\qquad \text{on}\quad \Omega,
 \end{cases}
\end{equation}
where $u_i := u_i(t,x)\ge0$ for $i=1,2,3,4$ denotes non-negative concentrations at time $t$ and 
position $x$ of four species $\mathcal{A}_i$ subject to non-negative initial concentrations $u_{i0}(x)\ge0$.   
Moreover, $d_i>0$ are the corresponding positive and constant diffusion coefficients. We suppose $x\in\Omega$, where $\Omega$ is a bounded domain of $\R^N$ ($N\ge 1$) with 
sufficiently smooth boundary $\partial\Omega\in C^{2+\alpha}$, $\alpha>0$. 
Finally, $n(x)$ is the outer normal unit vector at point $x$ of $\pa\Omega$ and we denote $\Omega_T=[0,T]\times\Omega$ and $\partial\Omega_T=[0,T]\times\partial\Omega$ for any $T>0$.

The above system \eqref{sys4} constitutes a mass-action law model of the evolution of a mixture of four diffusive species $\mathcal{A}_i, i=1,2,3,4$ 
undergoing the single reversible reaction 
\begin{equation*}
\mathcal{A}_1 + \mathcal{A}_2 \rightleftharpoons 
\mathcal{A}_3 + \mathcal{A}_4
\end{equation*}
until a unique positive detail balance equilibrium is reached in the large time behaviour, see Section \ref{Pre}.
For a derivation of 
system \eqref{sys4} or related mass-action law reaction-diffusion systems from kinetic models or fast reaction limits, we refer to \cite{BD,BCD,BP,DMP}.

Note that for the sake of readability, we have set the forward and backward reaction rates in \eqref{sys4} equal to one (the general case can be treated without any additional difficulty). Also, without loss of generality, we shall also assume that $\Omega$ is normalised (i.e. $|\Omega|=1$), which can always be achieved by rescaling the spatial variable.  
\bigskip


The quadratic reaction-diffusion system \eqref{sys4} and in particular the question of global-in-time solutions has lately received a lot of attention, see e.g. \cite{CDF,DFPV,GV}.

In \cite{DFPV}, for instance, a duality argument in terms of entropy density variables was 
used to prove the existence of global, non-negative $L^2$-weak solutions in any space dimension, see also Section \ref{Pre}.

While the existence of global classical solutions in 1D follows already 
from Amann, see e.g. \cite{Ama}, the existence of global classical solutions in 2D was recently shown by Goudon and Vasseur in \cite{GV} by using De Giorgi's method. For higher space dimensions the existence of classical solutions constitutes in general an open problem, for which the Hausdorff dimension of possible singularities was characterised in \cite{GV}. 

The (technical) criticality of quadratic nonlinearities was underlined by Caputo and Vasseur in \cite{CV}, where smooth solutions were shown to exist in any space dimension for systems with nonlinearities of sub-quadratic power law type, see also e.g. \cite{Ama}. 

A further related result by Hollis and Morgan \cite{HM} showed that 
if blow-up (here that is a concentration phenomenon since the total mass is conserved) occurs in system \eqref{sys4} in one concentration $u_i(t,x)$ at some time $t$ and position $x$, then at least one more concentration $u_{j \neq i}$ has also to blow-up (i.e. concentrate) at the same time and position. 
The proof of this result is based on a duality argument.

Recently in \cite{CDF}, an improved duality method allowed 
to show global classical solutions of system \eqref{sys4} in 2D 
in a significantly shorter and less technical way than De Giorgi's method.
\bigskip

In this work, we will prove in higher space dimensions $N\geq3$ and under a dimension-dependent closeness condition of the diffusion coefficients $d_i$ that the global, $L^2$-weak solutions to \eqref{sys4} (see \cite{DFPV}) are in fact global classical solutions. More precisely, we shall denote by 
\begin{equation}\label{delta}
\delta := b -a, \qquad\text{with}\quad b:=\sup_{i= 1..4} \{d_i\}, \quad a:=\inf_{i= 1..4} \{d_i\}>0,
\end{equation}
the maximal distance between the diffusion rates appearing in \eqref{sys4}
and prove the following:

\begin{The}[Global classical solutions and uniform-in-time bounds]\label{Theorem} \hfill\\
Consider $N\geq3$ and assume that $\delta$, $a$ and $b$ as given in \eqref{delta} satisfy
\begin{equation}\label{deltacon1}
\delta  < 2\,(C_{\frac{a+b}2, 2\Gamma})^{-1},\qquad\text{for}\quad
\Gamma>\frac{N+2}{2}-\frac{2N}{N+2}\ge\frac{13}{10},
\end{equation}
where $C_{\frac{a+b}2, 2\Gamma}$ is defined in Proposition \ref{propone}.
Assume initial data $u_{i0}\in L^{\frac{3N}{4}}(\Omega)$ (which also implies with $\frac{3N}{4}>2\Gamma$, assumption \eqref{deltacon1}, estimate \eqref{uapriori} of Proposition \ref{propone} and the theory of global, weak solutions in \cite{DFPV}, the existence of global, weak solutions $u_i\in L^{2\Gamma}(\Omega_T)$ to \eqref{sys4} for any $T>0$). 

Then, any global, weak solutions $u_i\in L^{2\Gamma}(\Omega_T)$ for any $T>0$ is a global, classical solutions to  \eqref{sys4}. Moreover, these solutions are bounded uniformly-in-time
in the sense that for all $\tau>0$, there exists a positive constant $C=C(\tau)$ such that  
\begin{equation}\label{Linftybound}
\|u_i(t,\cdot)\|_{L^{\infty}(\Omega)} \le C(\tau), \qquad t\ge\tau>0,\qquad \text{for all} \quad 
i=1,2,3,4.
\end{equation}

\end{The}
\medskip

\begin{Rem}\label{rem1}
The actual novelty of Theorem \ref{Theorem} concerns the range $\Gamma\in(\frac{N+2}{2}-\frac{2N}{N+2},\frac{N+2}{2})$ because when $\Gamma\ge\frac{N+2}{2}$, then the quadratic right hand side terms of \eqref{sys4} is bounded in $L^{\frac{N+2}{2}}(\Omega_T)$ and standard parabolic regularity estimates (which reflect the convolution with the heat kernel being in $L^{\frac{N+2}{N}-\eps}$, for all $\eps>0$) implies $u_i\in L^\infty(\Omega_T)$ and thus classical solutions, see e.g. Proposition \ref{propthree} from \cite{CDF}, where the more stringent $\delta$-smallness condition \eqref{deltacon} was assumed, which yields directly $u_i\in L^{{N+2}}(\Omega_T)$. 

In this paper, we are able to show classical solutions under the weaker $\delta$-smallness condition \eqref{deltacon1}
by combining regularity and bootstrap arguments in Bochner spaces with interpolation of Bochner spaces (see  Lemma \ref{LLInterpol}) 
and a further bootstrap based on a weak comparison argument (Lemma \ref{weakcomparisonbootstrap}).

For example, for $N=3$, our approach lowers the initially required regularity $u_i\in L^{2\Gamma}(\Omega_T)$ from 
the seemingly natural assumption $\Gamma = \frac{5}{2}$ 
(as assumed in Proposition \ref{propthree} from \cite{CDF}) down to $\Gamma>\frac{13}{10}$. Since global $L^2$-weak  solutions to \eqref{sys4} 
exist in all space dimensions, Theorem \ref{Theorem} leaves the gap $1<\Gamma\le  \frac{13}{10}$ in the case $N=3$, 
for which it remains an open problem if global weak solutions $u_i\in L^{2\Gamma}(\Omega_T)$ are also global classical solutions. This corresponds to systems \eqref{sys4}, for which 
$$
\delta  > 2\,(C_{\frac{a+b}2, 2\Gamma})^{-1}, \qquad\text{for all } \quad
\Gamma>\frac{N+2}{2}-\frac{2N}{N+2}.
$$
\end{Rem}
\medskip
\begin{Rem}
We remark that the assumption on the initial data $u_{i0}\in L^{\frac{3N}{4}}(\Omega)$ is not optimal, but chosen
for sake of the readability of the proof of Theorem \ref{Theorem}. In fact, the proof of Theorem \ref{Theorem} would work equally for $u_{i0}\in L^{\mu_0}(\Omega)$ with $\mu_0>\max\Bigl\{\nu,\frac{(\nu-1)N}{2}\Bigr\}$. 
\end{Rem}
\medskip
\begin{Rem}
We further remark that the uniform-in-time $L^{\infty}$-bound \eqref{Linftybound} follows from an interpolation argument between 
the exponential convergence to equilibrium of system \eqref{sys4} (see Proposition \ref{Convergence} below) 
and the fact that the proof of Theorem \ref{Theorem} only involves regularity constants, which grow not faster than 
polynomially in time. 
\end{Rem}
\medskip

\begin{Rem}
There is an alternative argument to derive the initially required regularity $u_i\in L^{2\Gamma}(\Omega_T)$
provided a suitable $\delta$-smallness condition of the diffusion coefficients $d_i$ of system \eqref{sys4}.  
System \eqref{sys4} implies, for instance that
\begin{equation*}
\pa_t (u_1+u_3) - d_1\Delta ( u_1+u_3)=  (d_3-d_1) \Delta u_3. 
\end{equation*} 
Then, by a duality estimate (see e.g. the proof of \cite[Lemma 3.4]{P2}), it follows directly for any $p\in (1,+\infty)$ and some  
$C=C(p,T)$ that 
$$ 
\|u_1+u_3\|_{L^p(\Omega_T)}\le C\left[1+|d_3-d_1|\|u_3\|_{L^p(\Omega_T)}\right].
$$
Thus, with $\|u_3\|_{L^p(\Omega_T)}\le \|u_1+u_3\|_{L^p(\Omega_T)}$, the required estimate for $\|u_3\|_{L^p(\Omega_T)}$ holds provided that $|d_3-d_1|$ is small enough.
\end{Rem}
\medskip

We emphasise that our proof of Theorem \ref{Theorem} only relies on the closeness assumption \eqref{deltacon1}, uniform-in-time $L^1$ a-priori estimates (which follows typically from mass conservation laws) 
and the quadratic order of the non-linearity of the reaction terms on the r.h.s. of \eqref{sys4}. 
Thus our results can be similarly applied to related mass-conserving reaction-diffusion systems with quadratic non-linearities. 
Moreover, one can also easily generalise the proof to systems with super-quadratic reaction terms provided the more stringent closeness condition \eqref{deltacon2}, see the following Corollary.

\begin{cor}  \label{Corollary}
Let $N\ge3$ and $\nu\ge2$ and consider the following non-linear reaction-diffusion systems:
\begin{equation}\label{sys4GEN}
 \begin{cases}
 \pa_t u_i - d_i\, \Delta_x u_i = f_i(u), &\qquad \text{on}\quad \Omega_T,\\
  n(x)\cd\nx u_i(t,x) = 0, &\qquad \text{on}\quad \partial\Omega_T,\\
  u_i(0,x) = u_{i0}(x)\ge0, &\qquad \text{on}\quad \Omega,
 \end{cases}
\end{equation}
where $u=(u_1,\ldots,u_k)\ge0$ denotes $k\in\N_+$ non-negative concentrations and for all $i\in\{1,\ldots,k\}$ the nonlinearities $f_i(u)$ are positivity-preserving, polynomials of order $\nu$, i.e. $f_i(u)\le C \sum_{i=1}^{k} u_i^{\nu}$ for a constant $C>0$ and satisfy conservation laws of the following form: For all $i\in\{1,\ldots,k\}$, there exists a constant, positive vector $0\le (a^i_j)_{j\in \{1,\ldots,k\}}\in\R^k$ with $a^i_i> 0$ such that $\sum_{j=1}^k a^i_j f_j(u)\leq0$ and thus
\begin{equation}\label{MGEN}
\forall t \ge 0: \quad\ a^i_i \|u_i(t,\cdot)\|_{L^1(\Omega)}  = a^i_i \iO u_i(t,x)\,dx \le \iO \sum_{j=1}^k a^i_j u_{j0}(x)\,dx \le M <+\infty, \qquad\forall \ i\in  \{1,\ldots,k\}.
\end{equation}

Assume that for $b:=\sup_{i= 1..k} \{d_i\}$, $a:=\inf_{i= 1..k} \{d_i\}>0$ and $\delta = b-a$ holds
\begin{equation}\label{deltacon2}
\delta  < 2\,(C_{\frac{a+b}2, \nu\Gamma})^{-1},\qquad\text{for}\quad
\Gamma>\frac{(\nu-1)N^2+4}{2 (N+2)},
\end{equation}
where $C_{\frac{a+b}2, 2\Gamma}$ is defined in Proposition \ref{propone}.
Assume initial data $u_{i0}\in L^{\frac{(2\nu-1)N}{4}}(\Omega)$ (which also implies with $\frac{(2\nu-1)N}{4}>2\Gamma$, assumption \eqref{deltacon2}, estimate \eqref{uapriori} of Proposition \ref{propone} and a theory of global, weak solutions analog to \cite{DFPV}, the existence of global, weak solutions $u_i\in L^{2\Gamma}(\Omega_T)$ to \eqref{sys4} for any $T>0$). 

Then, any global, weak solutions $u_i\in L^{2\Gamma}(\Omega_T)$ for any $T>0$ is a global, classical solutions to  \eqref{sys4}.
\end{cor}

\begin{Rem}
We remark that in contrast to Theorem \ref{Theorem}, Corollary \ref{Corollary} does not 
state a uniform-in-time bound analog to \eqref{Linftybound}. In fact, general systems \eqref{sys4GEN} will not feature an entropy functional like system \eqref{sys4}. Without the entropy method proving exponential convergence to equilibrium, our method only shows that global classical solutions to \eqref{sys4GEN} grow (at most) polynomially-in-time. More precisely, it is due to the uniform-in-time $L^1$-bounds \eqref{MGEN} that we are able to show global solutions with polynomially growing bounds for 
quadratic and super-quadratic systems.

For the existence of classical solutions to systems with sub-quadratic nonlinearities (i.e. $\nu<2$), we refer to Caputo and Vasseur \cite{CV} and the references therein. 
\end{Rem}

{\begin{Rem}\label{remrem}
Again, like in Remark \ref{rem1}, the novel range is $\Gamma\in(\frac{(\nu-1)N^2+4}{2 (N+2)},\frac{N+2}{2})$ because when $\Gamma\ge\frac{N+2}{2}$, then $L^{\infty}$-bounds and the existence of classical solutions follows from standard arguments. We remark that depending on the polynomial order $\nu$, Corollary \ref{Corollary} only constitutes an improvement under to following condition on the dimension $N$:
$$
\frac{(\nu-1)N^2+4}{2 (N+2)}<\frac{N+2}{2} \qquad\Longrightarrow\qquad
\nu<2+\frac{4}{N},
$$
which is false for $\nu$ large. Thus, our approach does not lead to improvements in the case of strongly super-quadratic nonlinearities. However, for $N=3$, we are able to lower the threshold, which is sufficient to show the existence of classical solutions, for systems with third order and slightly higher nonlinearities.   
\end{Rem}}

\bigskip
\underline{Notation:} Besides standard notations, we shall denote by 
$L^{p,q}(\Omega_T)$ the Bochner space with space-time norm:
$$
\|u\|_{p,q}=\left(\int_0^T\|u\|_{L^p(\Omega)}^{q}\,dt\right)^{1/q}.
$$
\bigskip

\underline{Idea of the proofs:}
For the convenience of the reader we present a short outline of the proofs of Theorem \ref{Theorem} and Corollary \ref{Corollary}:
First, we use an improved duality estimate as proven in \cite{CDF} in order to derive the starting a-priori estimate 
$u_i\in L^{2\Gamma}(\Omega_T)$ (or $u_i\in L^{\nu\Gamma}(\Omega_T)$ respectively). Then, we apply Sobolev's embedding theorem, interpolation with the uniform-in-time $L^1$-bound and other classical inequalities in order to derive  
$L^p$-estimates for some $p$ to be appropriately chosen. Next, we use these estimates as a starting point for a bootstrap scheme, which allows to improve the regularity of the solutions to $u_i\in L^{p_n,\infty}$, i.e. $u_i$ is $L^{\infty}$ in time with values $L^{p_n}$ in space. Then, after bootstrapping $p_n$ sufficiently large, the regularity $u_i\in L^{{p_n},\infty}$ allows to use the weak comparison Lemma \ref{weakcomparisonbootstrap}, which leads to global-in-time $L^{\infty}$ and, thus, classical solution to system \eqref{sys4}. Finally, uniform-in-time bounds for solutions of systems \eqref{sys4} follow from an interpolation argument with the exponential convergence to equilibrium.  
\medskip


\underline{Outline:} In Section \ref{Pre}, we shall first state basic properties of the systems \eqref{sys4} and recall previous existence results and methods. Then, in Section \ref{Global}, we shall first prove several Lemmas required for the proof of Theorem \ref{Theorem} and finally prove Corollary \ref{Corollary}. In the Appendix \ref{app}, we provide a proof of Lemma \ref{LLInterpol} for the convenience of the reader.

\section{Preliminaries}\label{Pre}

\subsubsection*{Non-negativity, mass conservation laws and equilibrium}
 
The chemical reaction terms on the r.h.s. of \eqref{sys4} satisfy the quasi-positivity property and thus ensure that solutions to \eqref{sys4} propagate 
the non-negativity of the initial data $u_{i0}\ge0$, i.e. for all $T>0$, we have
$$
u_i(t,x) \ge0, \qquad \text{on}\quad \Omega_T, \quad \text{for all} \quad 
i=1,2,3,4.
$$  
\medskip
 
Moreover, system \eqref{sys4} subject to homogeneous Neumann boundary conditions satisfies the following mass conservation laws:
\begin{equation}\label{cna}
\forall t\ge 0: \qquad M_{jk}:= \int_{\Omega} \left( u_j(t,x) + u_k(t,x) \right)\, dx = 
\int_{\Omega} \left( u_j(0,x) + u_k(0,x) \right)\, dx,
 \end{equation}
 where $j=\{1,2\}$ and $k=\{3,4\}$. Notice that
 only three of these four conservation laws are linearly independent. 
 
 The non-negativity of the solutions together with the mass conservation laws 
 \eqref{cna} provide natural uniform-in-time a-priori $L^1$-estimates for the 
 concentrations, i.e. 
 \begin{equation}\label{M}
 \sup_{t\ge 0} \|u_i(t,\cdot)\|_{L^1(\Omega)}\le M, \qquad \text{for all} \quad 
i=1,2,3,4,
 \end{equation} 
where $M$ denotes the total initial mass $M = M_{13} + M_{24} = M_{14} + M_{23}$.  
 \medskip
 
System \eqref{sys4} 
features a unique positive detail balance equilibrium.
Due to the homogeneous Neumann boundary conditions this equilibrium $(u_{i,\infty})_{i=1,..,4}$ consists of the unique positive constants balancing the reversible reaction 
$u_{1,\infty}\,u_{2,\infty}=u_{3,\infty}\,a_{u,\infty}$ and satisfying the conservation laws $u_{j,\infty} + u_{k,\infty} = M_{jk}$ for $(j,k)\in (\{1,2\},\{3,4\})$, 
that is:
\begin{equation}\label{ABCDEqu}
\begin{cases}
u_{1,\infty} = \frac{M_{13} M_{14}}{M}, & \quad u_{3,\infty} = M_{13}-\frac{M_{13} M_{14}}{M} = \frac{M_{13} M_{23}}{M}, \\
u_{2,\infty} = \frac{M_{23} M_{24}}{M}, & \quad u_{4,\infty} = M_{14}-\frac{M_{13} M_{14}}{M} = \frac{M_{14} M_{24}}{M}. 
     \end{cases}
\end{equation}

\subsubsection*{Duality estimates, entropy variables and global weak solutions}

The system \eqref{sys4} can also be rewritten in terms of the entropy density variables 
$z_i:=u_i \log(u_i) - u_i$ (compare with the entropy functional \eqref{ABCDXEntr} below). By introducing the sum $z:=\sum_{i=1}^4 z_i$, it holds (with $a$ and $b$ defined as in \eqref{delta}) that
\begin{equation}\label{sysentr}
\begin{cases}
\partial_t z - \Delta_x (A\,z) \le 0, &\qquad \text{on}\quad \Omega_T,\qquad  A(t,x):= \frac{\sum_{i=1}^4 d_i \,z_i}{\sum_{i=1}^4 z_i} \in [a,b],\\[1mm] 
n(x)\cd\nx z(t,x) = 0, &\qquad \text{on}\quad \partial\Omega_T,\\[1mm]
z(0,x)=\sum_{i=1}^4 z_{i0}(x), &\qquad \text{on}\quad \Omega,
\end{cases}
\end{equation} 
Then, by a duality argument (see e.g. \cite{DFPV,HM,Pdual} and the references therein), the parabolic problem 
\eqref{sysentr} satisfies for all $T>0$ and for all space dimensions $N\ge1$ the following a-priori estimate
\begin{equation}\label{dual}
\|z_i\|_{L^{2}(\Omega_T)} \le C(1+T)^{1/2} \,   \bigg\|\sum_{i=1}^4 u_{i0}(\log(u_{i0})-1)\bigg\|_{L^2(\Omega)},\qquad i=1,..,4,
\end{equation}
where $C$ is a constant independent of $T$, see \cite{CDF,DFPV}. Thus, given $(u_{i0})_{i=1,..,4}\in L^2 (\log L)^2(\Omega)$, we have $(u_i)_{i=1,..,4}\in L^2 (\log L)^2(\Omega_T)$ and the quadratic nonlinearities on the right hand side of \eqref{sys4} are uniformly integrable, which allows to prove the existence of global $L^2$-weak solutions in all space dimensions $N\ge1$, see \cite{DFPV}. Moreover, in 2D and in higher space dimension under the assumption of sufficiently strong $\delta$-smallness condition, the following duality Lemma allows to show global classical solutions:
\begin{prp}[\cite{CDF}]
  \label{propone} 
  Let $\Omega$ be a bounded domain of $\R^N$ with smooth
  (e.g. $C^{2+\alpha}$, $\alpha>0$) boundary $\partial\Omega$ and $T>0$.
 We consider a coefficient function $M :=
  M(t,x)$ satisfying for $p\in (2,\infty)$
 $$
    0 < a \leq M(t,x) \leq b < +\infty
    \qquad \text{ for } (t,x) \in \Omega_T.
$$
Then, any function $u$ satisfying
  \begin{equation}
    \label{eq:heat-variable-forward}
    \left\{
      \begin{aligned}
        &\partial_t u - \Delta_x (M u) \le 0, &&\quad \text{ on } \quad \Omega_T,
        \\
        &u(0,x) = u_0(x)\in
  L^p(\Omega), &&\quad \text{ for } \quad x \in \Omega,
        \\
        &\nabla_x u \cdot \nu(x) = 0, &&\quad \text{ on } \quad [0,T]
        \times \partial \Omega ,
      \end{aligned}
    \right.
  \end{equation}
and
  \begin{equation}
    \label{eq:const-small-q}
    C_{\frac{a+b}2,p'} \, \frac{b-a}{2} < 1,
  \end{equation}
  (with $p'<2$ denoting the H\"older conjugate of $p$)
  satisfies for all $T>0$ 
$$
    \|u\|_{L^p(\Omega_T)}
    \leq
    C_T \,
    \|u_0\|_{L^p(\Omega)}.
$$
Here, the constant $C_T$ depends polynomially on $T$ and 
  the constant $C_{m,q}>0$ in \eqref{eq:const-small-q} is defined for $m>0$, $p' \in (1,2)$
  as the best (that is, smallest) constant in the parabolic regularity
  estimate
$
    \| \Delta_x v\|_{L^q(\Omega_T)} \le C_{m,q}\,  \|f\|_{L^q(\Omega_T)},
$
  where  $v:[0,T] \times \Omega \to \R$ is the solution
  of the backward heat equation with homogeneous Neumann boundary conditions:
  \begin{equation}
    \label{quinze}
    \left\{
      \begin{aligned}
        &\partial_t v + m \,\Delta_x v = f, &&\quad \text{ on } \quad \Omega_T,
        \\
        &v(T,x) = 0, &&\quad \text{ for } \quad x \in \Omega,
        \\
        &\nabla_x v \cdot \nu (x) = 0, &&\quad \text{ on } \quad  [0,T]
        \times \partial \Omega.
      \end{aligned}
    \right.
  \end{equation}
  
Specifically for system \eqref{sys4}, we can rewrite the system in terms of 
$z =  u_1+u_2+u_3+u_4$ (it is not even necessary to consider entropy density variables) 
\begin{equation}\label{sysdual}
\begin{cases}
\partial_t z - \Delta_x (A\,z) = 0, &\qquad \text{on}\quad \Omega_T, \qquad A(t,x):= \frac{\sum_{i=1}^4 d_i \,u_i}{\sum_{i=1}^4 u_i} \in [a,b],\\
 n(x)\cd\nx z(t,x) = 0,&\qquad \text{on}\quad \partial\Omega_T, \\
z(0,x)=\sum_{i=1}^4 u_{i0}(x), &\qquad \text{on}\quad \Omega,
\end{cases}
\end{equation} 
where $a$ and $b$ are defined as in \eqref{delta} and obtain for all $T>0$ 
\begin{equation}\label{uapriori}
\forall i = 1,2,3,4: \qquad \|u_i\|_{L^p(\Omega_T)}    \leq C_T \, \|z(0,\cdot)\|_{L^p(\Omega)},
\qquad \text{provided that}\quad
C_{\frac{a+b}2,p'} \, \frac{b-a}{2} < 1.
\end{equation}
\end{prp}

\subsubsection*{Entropy functional, entropy method and exponential convergence to equilibrium}

The detail balance structure of system \eqref{sys4} ensures also the following non-negative entropy 
functional $E((u_i)_{i=1,..,4})$ and the associated entropy dissipation 
functional
$D((u_i)_{i=1,..,4})=-\frac{d}{dt}E((u_i)_{i=1,..,4})$:
\begin{align}
E((u_i)_{i=1,..,4})(t) =&\ \sum\limits_{i=1}^{4} \iO \Big(u_i(t,x)\log(u_i(t,x))-u_i(t,x)+1\Big)\,dx,\label{ABCDXEntr}\\
D((u_i)_{i=1,..,4})(t) =&\ \sum\limits_{i=1}^4 \iO 4 \,d_i\,|\nabla_x \sqrt{u_i(t,x)}|^2\,dx 
+ \iO (u_1\,u_2 - u_3\,u_4) \log \left(\frac{u_1\,u_2}{u_3\,u_4}\right)(t,x)\,dx.
\label{EntrDiss}
\end{align}
It is easy to verify that the following entropy dissipation law holds
(for sufficiently regular solutions $(u_i)_{i=1,..,4}$ of (\ref{sys4})) 
for all $t\ge 0$
\begin{equation*} \label{entlaw}
E((u_i)_{i=1,..,4})(t) + \int_0^t D((u_i)_{i=1,..,4})(s)\, ds = E((u_{i0})_{i=1,..,4}).
\end{equation*}


In \cite{DF3}, exponential 
convergence in $L^1$ towards the unique constant equilibrium \eqref{ABCDEqu} (with explicitly computable rates) was shown for global $L^2$-weak solutions in all space dimensions $N$.
The proof of this statement was based on the so called entropy method,  
where a quantitative entropy entropy-dissipation estimate of the form 
\begin{equation}\label{EED}
D((u_i)_{i=1,..,4}) \ge C \left [ E((u_i)_{i=1,..,4}) - E((u_{i,\infty})_{i=1,..,4})\right]
\end{equation}
with an explicitly computable constant $C$ was established, which uses only natural a-priori bounds of the system 
and thus significantly improved the results of \cite{DF2}. 

The following Proposition \ref{Convergence} recalls the corresponding results of  \cite{DF3}:


\begin{prp}[Exponential convergence to the equilibrium, \cite{DF3}] \label{Convergence}\hfill\\
Let $\Omega$ be a bounded domain with sufficiently smooth boundary such that Poincar\'e's and  Sobolev's inequalities and thus the Logarithmic Sobolev inequality
hold. 
Let $(d_i)_{i=1,..,4}>0$ be positive diffusion coefficients. Let
the 
initial data $(u_{i,0})_{i=1,..,4}$ be nonnegative functions of $L^{2}\,(\log L)^2(\Omega)$ with positive masses $(M_{jk})_{(j,k)\in (\{1,2\},\{3,4\})}>0$ (see 
\eqref{cna}). 

Then, any global solution $(u_i)_{i=1,..,4}$ of \eqref{sys4}, which satisfies the entropy dissipation law 
\eqref{entlaw} (this is true for any weak or classical solutions as shown to exist in \cite{CDF,DFPV}) decays exponentially towards the positive equilibrium state 
$(u_{i,\infty})_{i=1,..,4}>0$ defined by \eqref{ABCDEqu}:
\begin{equation*}
\sum_{i=1}^{4}\| u_i(t,\cdot) - u_{i,\infty}\|_{L^1(\Omega)}^2  
\le C_1 \Big(E((u_{i,0})_{i=1,..,4}) - E((u_{i,\infty})_{i=1,..,4})\Big)\,e^{-C_2\,t}, 
\end{equation*}
for all $t\ge0$ and for constants $C_1$ and $C_2$, which can be explicitly computed. 
\end{prp}

The following Proposition \ref{propthree} from \cite{CDF} shows the consequences of the Propositions \ref{propone} and \ref{Convergence}
specific to the system \eqref{sys4}.

\begin{prp}[\cite{CDF}]\label{propthree}
Let $\Omega$ be a bounded domain of $\R^N$ with smooth
(e.g. $C^{2+\alpha}$, $\alpha>0$) boundary $\partial\Omega$.  For
all $i=1,\ldots,4$ assume positive diffusion coefficients $d_i>0$ such
that $0<a = \inf \{d_i\}_{i=1..4}$, $b=\sup \{d_i\}_{i=1..4}$ and
nonnegative initial data $u_{i0} \in L^{\infty}(\Omega)$ ($i=1..4$). 

Then, if 
\begin{equation}\label{deltacon}
\delta= b-a < 2\,(C_{\frac{a+b}2, \frac{N + 2}{N}})^{-1},
\end{equation}
 there
exists a nonnegative weak solution $u_i \in L^{\infty}([0, +\infty)
\times \Omega)$ to system
\eqref{sys4} subject to the initial data $u_{i0}$. These $L^{\infty}$ solutions are in fact classical solutions, which follows from  standard parabolic regularity. 

Moreover,  there exist two constants $C_1, C_2>0$ such
that
  \begin{equation}\label{quatrebis}
    \forall t \ge 0: \qquad \qquad
    \sum_{i=1}^4 \| u_i(t,\cdot) - u_{i\infty}\|_{L^{\infty}(\Omega)}
    \le C_1\, e^{- C_2\, t}.
  \end{equation}
Here, the norm $\|u_i\|_{L^{\infty}([0, +\infty) \times\Omega)}$ and
the constants $C_1, C_2$ can be explicitly bounded in
terms of the domain $\Omega$, space dimension $N$, the initial norms
$\|u_{i0}\|_{L^{\infty}(\Omega)}$  and the
diffusion coefficients $d_i$, $(i=1..4)$. 
\end{prp}

\section{Existence of Global Classical Solutions}\label{Global}
 
In this Section, we first prove the existence of global classical solutions of the system \eqref{sys4}
under the $\delta$-smallness assumption 
\eqref{deltacon1}, where $\delta$ (as defined in \eqref{delta}) measures the maximal distance of the diffusion coefficients
of systems \eqref{sys4} on a domain $\Omega\subset\R^N$ for a given space dimension $N\ge 3$ 
(recall that global classical solutions of system \eqref{sys4} are known for $N=1,2$). 
As a consequence of the $\delta$-smallness assumption 
\eqref{deltacon1}, estimate \eqref{uapriori} of Proposition \ref{propone} provides an $u_i\in L^{2\Gamma}(\Omega_T)$-estimate for the concentration $u_i$ of \eqref{sys4}. In particular, we are interested in the range 
\begin{equation}
\label{pbounds}
\frac{N+2}{2}-\frac{2N}{N+2}<\Gamma< \frac{N+2}{2},
\end{equation}
since for $\Gamma\ge N+2$ and the considered quadratic nonlinearities $f \le |u_1\, u_2 - u_3\,u_4| \in L^{\Gamma}(\Omega_T)$, a standard parabolic regularity bootstrap argument for the heat equation with Neumann boundary conditions, i.e.
\begin{equation}
    \label{eq:heat-Neumann}
    \left\{
      \begin{aligned}
        &\partial_t u - d \,\Delta_x u = f &&\quad \text{ on } \quad \Omega_T,
        \\
        &u(0,x) = u_0(x) &&\quad \text{ for } \quad x \in \Omega,
        \\
        &\nabla u \cdot \nu(x) = 0 &&\quad \text{ on } \quad [0,T]
        \times \partial \Omega,
      \end{aligned}
    \right.
  \end{equation}
implies  for a right-hand-side $f\in L^\frac{N+2}{2}(\Omega_T)$ 
that the solution satisfies the following $L^{\infty}$ estimates 
$$
\|u\|_{L^{\infty}(\Omega_T)}\le C_{T},
$$ 
where $C_T$ grows at most polynomially w.r.t. $T$, (see \cite{CDF} for the polynomial dependence on $T$). In particular, 
the range $\Gamma \ge \frac{N+2}{2}$ was considered in Proposition \ref{propthree}, where $\delta$ was assumed to satisfy the more stringent $\delta$-smallness assumption \eqref{deltacon}, which yields $f\in L^\frac{N+2}{2}(\Omega_T)$. 
\bigskip 
 
 Thus, we consider here the $\delta$-smallness assumption 
\eqref{deltacon1} and estimate \eqref{uapriori} of Proposition \ref{propone} yields a-priori estimate of the form  
\begin{equation}\label{Lpg2}
\|u_i\|_{L^{2\Gamma,2\Gamma}(\Omega_T)}\leq C_T \, \|u_{i0}\|_{L^{2\Gamma}(\Omega)},\qquad \text{for some} \quad \Gamma\in\Bigl(\frac{N+2}{2}-\frac{2N}{N+2},\frac{N+2}{2}\Bigr),
\end{equation}
where $C_T$ is a constant depending polynomially on the time $T$ and $\frac{N+2}{2}-\frac{2N}{N+2}\ge\frac{13}{10}$ for $N\ge3$. 
\bigskip


As a first step in proving Theorem \ref{Theorem}, we are faced with the fact that duality estimates like \eqref{Lpg2} 
address spaces $L^{ 2\Gamma,2\Gamma}(\Omega_T)$, which feature equal time- and space-integrability.
However, as we shall see in the following, bootstrapping the problem \eqref{eq:heat-Neumann} naturally leads to 
Bochner spaces with higher time- than space-integrability. 

In this Section, we present a bootstrap, which allows to conclude from $u_i\in L^{2\Gamma,2\Gamma}$ to  $u_i\in L^{\frac{3N}{4},\infty}$.
Furthermore, we shall take care to ensure that all involved constants grow at most polynomially in $T$.
\smallskip
\begin{Lem}\label{Lemma31}
Let $N\ge3$ and $u$ satisfy the problem \eqref{eq:heat-Neumann}, the following uniform-in-time $L^1$-bound \eqref{MM}
\begin{equation}\label{MM}
 \sup_{t\ge 0} \|u(t,\cdot)\|_{L^1(\Omega)}\le M 
\end{equation}
(for instance as a consequence of a suitable mass conservation law) and consider initial data $u_0\in L^p(\Omega)$ 
for a $p>1$ to be chosen.

Then, (by using Sobolev's embedding theorem, the uniform-in-time bound \eqref{MM} as well as H\"older's and Young's inequality) $u$ satisfies the following $L^p$ estimate for all $T>0$:
\begin{equation}\label{h2}
  \|u(T)\|_{p}^{p} 
 +  \tilde{C} \|u\|_{q,p}^p
  \le 
 p\|f\|_{\mu,\gamma}
\|u\|_{\mu'(p-1),\gamma'(p-1)}^{p-1}
+  \|u_0\|_{p}^{p} 
 +C'T.\end{equation}
where $q=q(p,N):=p\frac{N}{N-2}>p$ for $N\ge3$ and 
the constants $C'$ and $\tilde{C}$ only depend on $p$, $M$, $d$, $N$  and $C_s$, which is the constant from the Sobolev's embedding theorem. Moreover, $\gamma$, $\gamma'$, $\mu$ and $\mu'$ are H\"older conjugates to be chosen.

Furthermore, by using once more the uniform-in-time bound \eqref{MM} and provided that $1\le \mu'(p-1)\le q$ holds, $u$ also satisfies
 \begin{equation}\label{h2new}
  \|u(T)\|_{p}^{p} 
 +  \tilde{C} \|u\|_{q,p}^p
  \le 
 C(p,M)\|f\|_{\mu,\gamma}
\left(\int_0^T\|u\|_q^{q\frac{\gamma'}{\mu'}\frac{\mu'(p-1)-1}{(q-1)}}dt\right)^{1/\gamma'}
+  \|u_0\|_{p}^{p} 
 +C'T.	
\end{equation}
\end{Lem}

\begin{proof}

By testing \eqref{eq:heat-Neumann} with 
$p\,|u|^{p-1}\mathrm{sgn}(u)$ (more precisely by testing with a smoothed version of the modulus $|u|$ and its derivative $\mathrm{sgn}(u)$ and letting then the smoothing 
tend to zero), we obtain by integration-by-parts and with the constant $C_0(p,d):=\frac{4(p-1)\,d}{p}$~:
\begin{equation*}
 \frac{d}{dt}  \|u\|_{p}^{p} 
 +  C_0(p,d)\iO \left|\nabla_x (u^{p/2})\right|^2dx 
 \le 
 p\int_\Omega |f| |u|^{p-1}\,dx.
\end{equation*}
\begin{flushleft}
{\it  Sobolev's embedding theorem}
\end{flushleft}
After adding on both sides $C_0\int_\Omega |u|^pdx$, we apply  Sobolev's embedding theorem for $N\ge3$, i.e. 
$$\|u^{p/2}\|_{H^1(\Omega)}^2
=
\iO \left|\nabla_x (u^{p/2})\right|^2dx
+\iO \left|u\right|^pdx
\geq C_s\|u\|_{q}^{p} $$
with Sobolev constant $C_s$ and 
$$
q=q(p,N):=p\frac{N}{N-2}=p+\frac{2p}{N-2}\qquad\text{and}\quad  q(p)>p\quad \text{for all}\ N\ge3.
$$
Therefore, we get:
\begin{equation}\label{h0}
 \frac{d}{dt}  \|u\|_{p}^{p} 
 +  C_0C_s \|u\|^p_{q}
 \le 
 p\int_\Omega |f| |u|^{p-1}\,dx
 +C_0\|u\|_{p}^{p}.
\end{equation}

\begin{flushleft}
{\it  Interpolation and the mass conservation property}
\end{flushleft}

 Here we remark that integration of \eqref{h0} by means of a Gronwall type argument would lead to  global, yet exponentially growing estimates of $\|u\|_p^p$. However, such exponential growing estimates can be avoided (and should be avoided in order to retain the possibility of interpolating a-priori estimates with exponential convergence to equilibrium) 
by using the mass conservation laws \eqref{cna}, which provides a uniform-in-time $L^1(\Omega)$ bound of the form \eqref{MM}. 
More precisely, we interpolate the $L^p$ term $\|u\|_{p}^{p}$ on the right hand side of \eqref{h0} between $L^1$ and $L^q$, i.e. $\frac{1}{p}=\frac{\theta}{1} +\frac{1-\theta}{q}$ such that 
$\theta=\frac{q-p}{q-1}\frac1p\in(0,1)$ since $1<p<q$. Thus, due to the uniform-in-time bound 
$\|u(t,\cdot)\|_1\le M$ for all $t\ge 0$, we obtain
\begin{equation}\label{h1ii}
 \frac{d}{dt}  \|u\|_{p}^{p} 
 +  C_0C_s \|u\|^p_q
 \le 
 p\int_\Omega |f| |u|^{p-1}\,dx
 +C_0M^{\frac{q-p}{q-1}}\|u\|_{q}^{s},
\end{equation}
with an exponent
$$
s=s(p,N):=\frac{q(p-1)}{q-1}=\frac{Np(p-1)}{2+N(p-1)} = p - \frac{2p} {2+N(p-1)}<p, \qquad \forall p>1,\ N\ge3.
$$
Now, since $p>s$, the last term on the right hand side of \eqref{h1ii} can be controlled by the 
term $\|u\|^p_q$ on the left hand side of \eqref{h1ii} as in the following step. 

\begin{flushleft}
{\it  Young's and H\"olders's inequality, time-integration and further H\"older inequality}
\end{flushleft}

Next, we apply Young's inequality to the last term on the right hand side of \eqref{h1ii} to derive for a $\delta>0$ to be chosen sufficiently small
$$
 \frac{d}{dt}  \|u\|_{p}^{p} 
 +  C_0C_s \|u\|^p_q
 \le 
 p\int_\Omega |f| |u|^{p-1}\,dx
 +C'(C_0,\delta,M)+\delta\|u\|_{q}^{p},
$$
and, thus
$$
 \frac{d}{dt}  \|u\|_{p}^{p} 
 +  \tilde{C} \|u\|^p_q
 \le 
 p\int_\Omega |f| |u|^{p-1}\,dx
 +C'(C_0,\delta,M),
$$
with $\tilde{C}=C_0C_s-\delta>0$ for small enough $\delta$.
\medskip

In the following, we estimate the first term on the right hand side with H\"olders's inequality. In fact,  this shall first be done for a general H\"older exponent $\frac{1}{\mu'}+\frac1\mu=1$, i.e. 
\begin{equation}\label{h1}
 \frac{d}{dt}  \|u\|_{p}^{p} 
 + \tilde{C} \|u\|_{q}^{p}
 \le 
 p\|f\|_{\mu}\|u\|_{\mu'(p-1)}^{p-1}
 +C'
\end{equation}
where $\mu$ can be chosen in order to match the necessary assumptions. 
\begin{Rem}\label{case2}
We remark at this point that we will mostly chose $\mu$ such that $\mu'(p-1)=q$, i.e. we shall choose $\mu$ according to the maximal space integrability, which can be controlled by the left hand side term $\|u\|_{q}^{p}$. 
\end{Rem}

 Next, we integrate over $(0,T)$ and apply  H\"older again:
$$
  \|u(T)\|_{p}^{p} 
 +  \tilde{C}\|u\|_{q,p}^p
  \le 
 p\biggl(\int_0^T\|f\|_{\mu}^{\gamma} \,dt\biggr)^{\!\frac{1}{\gamma}}\biggl(\int_0^T\|u\|_{\mu'(p-1)}^{\gamma'(p-1)}\,dt\biggr)^{\!\frac{1}{\gamma'}}
+  \|u_0\|_{p}^{p} 
 +C'T
 $$
where $1=\frac1\gamma+\frac{1}{\gamma'}$. Thus, we have
$$
  \|u(T)\|_{p}^{p} 
 +  \tilde{C} \|u\|_{q,p}^p
  \le 
 p\|f\|_{\mu,\gamma}
\|u\|_{\mu'(p-1),\gamma'(p-1)}^{p-1}
+  \|u_0\|_{p}^{p} 
 +C'T,
$$
which proves the estimate \eqref{h2}.

\begin{flushleft}
{\it Further interpolation with the uniform-in-time $L^1$-bound}
\end{flushleft}

	We first remind that 
	$$
	\|u\|_{\mu'(p-1),\gamma'(p-1)}^{p-1}=\left(\int_0^T\|u\|_{\mu'(p-1)}^{\gamma'(p-1)}dt\right)^{1/\gamma'}.
	$$
	In order to further exploit the mass conservation property in cases such that $1\le \mu'(p-1)\le q$, we will interpolate $L^{\mu'(p-1)}\hookrightarrow
	 L^q\cap L^1$ with a parameter $\theta\in[0,1]$ that satisfies 
	 $$\frac{1}{\mu'(p-1)}=\theta+\frac{1-\theta}{q},\quad\text{or}\quad\theta=\frac{q-\mu'(p-1)}{\mu'(p-1)(q-1)}.$$
	Thus, we get
	$$
	\|u\|_{\mu'(p-1)}\leq\|u\|_1^\theta\|u\|_q^{1-\theta}\leq M^\theta\|u\|_q^{1-\theta}
	$$
	and
	$$
	\|u\|_{\mu'(p-1),\gamma'(p-1)}^{p-1}
	\leq
	\left(\int_0^T(M^\theta\|u\|_q^{1-\theta})^{\gamma'(p-1)}dt\right)^{1/\gamma'}
	=
	M^{\theta(p-1)}	
	\left(\int_0^T\|u\|_q^{\gamma'(1-\theta)(p-1)}dt\right)^{1/\gamma'}
	$$
	and by substituting $1-\theta=\frac{q\mu'(p-1)-q}{\mu'(p-1)(q-1)}$ into the previous relation, we have:
\begin{equation*}
	\|u\|_{\mu'(p-1),\gamma'(p-1)}^{p-1}
	\leq
		M^{\theta(p-1)}	
	\left(\int_0^T\|u\|_q^{q\frac{\gamma'}{\mu'}\frac{\mu'(p-1)-1}{(q-1)}}dt\right)^{1/\gamma'}.
\end{equation*}
	Thus, altogether we have 
$$
  \|u(T)\|_{p}^{p} 
 +  \tilde{C} \|u\|_{q,p}^p
  \le 
 C(p,M)\|f\|_{\mu,\gamma}
\left(\int_0^T\|u\|_q^{q\frac{\gamma'}{\mu'}\frac{\mu'(p-1)-1}{(q-1)}}dt\right)^{1/\gamma'}
+  \|u_0\|_{p}^{p} 
 +C'T,	
$$  
which proves \eqref{h2new} with $C(p,M)=p M^{\theta(p-1)}$.
\end{proof}

The estimate \eqref{h2} still allows to choose $\mu$ and $\gamma$ in our goal of deriving a-priori estimates for $u$.
At this point, there are two basic options: First, we could aim to control the term $\|u\|_{\mu'(p-1),\gamma'(p-1)}^{p-1}$ entirely 
by the duality estimate \eqref{Lpg2}, which we have anyway already used to ensure the integrability $f\in L^{\Gamma,\Gamma}$ and $u_i\in L^{2\Gamma,2\Gamma}$. 
Thus, in this first case, we choose $\mu$ and $\gamma$ such that  
$$
\gamma=\mu=\Gamma\qquad\text{and}\qquad \gamma'(p-1)=\mu'(p-1)=2\Gamma,
$$
since the duality estimate \eqref{Lpg2} implies $\|u\|_{2\Gamma,2\Gamma}\leq C_T$. Thus, we choose 
$$\gamma=\mu=\frac{p+1}{2}$$
and \eqref{h2} becomes:
\begin{equation}\label{h2s}
  \|u(T)\|_{p}^{p} 
 +  \tilde{C} \|u\|_{q,p}^p
  \le 
 p\|f\|_{\frac{p+1}{2},\frac{p+1}{2}}
\|u\|_{p+1,p+1}
+  \|u_0\|_{p}^{p} 
 +C'T.\end{equation}
If  we consider $p=2$, then \eqref{h2s} requires to assume the $\delta$-closeness condition in Proposition \ref{propone} small enough such that 
the duality estimate ensures $u_i\in L_{3,3}$ and $f\in L_{\frac32,\frac32}$. In the case $N=3$, under such a $\delta$-closeness assumption, 
\eqref{h2s} then implies directly $u_i\in L_{2,\infty}$ and we can apply Lemma \ref{weakcomparisonbootstrap} in order to establish global classical solutions in 3D, see also \cite{LSY}. In order to show classical solutions in higher space dimensions $N>3$, we can proceed similar by considering \eqref{h2s} for suitable exponents $p>2$. However, the idea to directly control the term $\|u\|_{\mu'(p-1),\gamma'(p-1)}^{p-1}$ via a duality estimate \eqref{Lpg2} does not lead to such general results as stated in Theorem \ref{Theorem}. This which will become clear from the following considerations.
\medskip

From now onwards, we are interested in the second case, where the term $\|u\|_{\mu'(p-1),\gamma'(p-1)}^{p-1}$ cannot be controlled by  the duality estimate \eqref{Lpg2}, but has to be controlled by the term $\|u\|_{q,p}^p$ on the left hand side of \eqref{h2}.
In order to constitute an improvement of the argument described in the previous paragraph, we are interested in cases when 
$$\mu'(p-1)>2\Gamma$$
holds. 
\begin{Rem}
In fact, we will see in the following that our method naturally considers $\mu\leq\gamma$, i.e. that the time-integrability is larger (or at least equal) to the space-integrability. This leads to two possible subcases 
$$\gamma'(p-1)<2\Gamma<\mu'(p-1)\qquad\text{or}\qquad 2\Gamma< \gamma'(p-1)<\mu'(p-1),
$$
where the first subcase would in principle still allow to partially use the duality estimate $\|u\|_{2\Gamma,2\Gamma}\leq C_T$ in a bootstrap argument. 

However, by anticipating some details from the below calculations, in particular by considering the starting value of a bootstrap exponent $p_0=\frac{\Gamma(N+2)-2}{N-2}>\frac{N}{2}$ for $\Gamma>\frac{N+2}{2}-\frac{2N}{N+2}$ and further $\mu'(p_0-1)=\frac{p_0N}{N-2}$ and $\gamma'(p_0-1)=p_0$, we see that the below arguments are able to successfully treat some cases (in dimensions $N=3,4$), where
$$
2\Gamma< \gamma'(p_0-1)<\mu'(p_0-1).
$$
This means that the most general results as stated in Theorem \ref{Theorem} are derived from entirely controlling the term $\|u\|_{\mu'(p-1),\gamma'(p-1)}^{p-1}$ by the term $\|u\|_{q,p}^p$ on the left hand side of \eqref{h2}. 
\end{Rem}

\bigskip

The following two Lemmata \ref{equal} and \ref{unequal} develop further the estimates of Lemma \ref{Lemma31} for two different choices of $\mu$ and $\gamma$:
Lemma \ref{equal} takes into account that the duality estimates \eqref{Lpg2} only allow to control $f$ in spaces with the same regularity in space and in time, i.e. $f\in L^{\Gamma,\Gamma}(\Omega_T)$.
Therefore, Lemma \ref{equal} sets $\mu=\gamma$:

\begin{Lem}[$\mu=\gamma$ estimate]\label{equal}
Consider $N\ge3$ and assume that $f\in L^{\frac{pN+2}{N+2},\frac{pN+2}{N+2}}$ and $u_0\in L^p(\Omega)$ for some $p>1$.  

Then, we have with $q=\frac{pN}{N-2}$
\begin{equation}\label{apri1}
u \in L^{p,\infty}\cap L^{q,p} \qquad \text{with}\qquad 
 \|u(T)\|_{p}^{p}, \ \|u\|_{q,p}^p
  \le 
C(p,M,\|f\|_{\frac{pN+2}{N+2},\frac{pN+2}{N+2}},\|u_0\|^{p})
 \end{equation}
with a constant $C$, which grows at most polynomially in time $T$. 

	\end{Lem}
\begin{proof}
By recalling \eqref{h2new} from Lemma \ref{Lemma31} and setting 
$\gamma=\mu$, we have especially 
\begin{equation}\label{h2newnew}
  \|u(T)\|_{p}^{p} 
 +  \tilde{C} \|u\|_{q,p}^p
  \le 
 C(p,M)\|f\|_{\mu,\mu}
\left(\int_0^T\|u\|_q^{q\frac{\mu'(p-1)-1}{(q-1)}}dt\right)^{1/\mu'}
+  \|u_0\|_{p}^{p} 
 +C'T.
\end{equation}
Then, the first term on the right hand side can be controlled by the term $\|u\|_{q,p}^p$ on the left hand side of \eqref{h2new} by setting 
$$
q\frac{\mu'(p-1)-1}{(q-1)}=p, \qquad \text{such that} \qquad
\left(\int_0^T\|u\|_q^{q\frac{\mu'(p-1)-1}{(q-1)}}dt\right)=\|u\|_{q,p}^p,
$$
if we verify that the corresponding H\"older exponent satisfies $\mu'>1$ (such that $\frac{1}{\mu'}<1$) and that the interpolation $L^{\mu'(p-1)}\hookrightarrow L^q\cap L^1$ used in Lemma \ref{Lemma31} is admissible. This is done by recalling that $q=\frac{pN}{N-2}$ and straightforward calculations show the corresponding values of $\mu'$, $\mu$ and  $\theta$ to be
\begin{equation*}
\mu'=\frac{pq+q-p}{q(p-1)}=\frac{pN+2}{N(p-1)}>1,\quad	\quad 
\mu=\frac{pq+q-p}{2q-p}=\frac{pN+2}{N+2}>1,
\end{equation*}
$$
\theta=\frac{q-p}{pq+q-p}=\frac{2}{pN+2}\in(0,1),
$$
for all $p>1$.

Therefore, the estimate \eqref{h2newnew} becomes:
\begin{equation*}
  \|u(T)\|_{p}^{p} 
 +  \tilde{C} \|u\|_{q,p}^p
  \le 
C( p,M)\|f\|_{\frac{pN+2}{N+2},\frac{pN+2}{N+2}}
	\left(\|u\|_{q,p}^p\right)^{\frac{1}{\mu'}}
+  \|u_0\|_{p}^{p} 
 +C'T.\end{equation*}
Next, since $\mu',\mu>1$, we estimate $\left(\|u\|_{q,p}^p\right)^{\frac{1}{\mu'}}$ with Young's inequality and obtain for  a sufficiently small enough $\epsilon>0$
\begin{equation}\label{h2myy}
  \|u(T)\|_{p}^{p} 
 +  \hat{C} \|u\|_{q,p}^p
  \le 
C(p,M,\epsilon)\|f\|^{\frac{pN+2}{N+2}}_{\frac{pN+2}{N+2},\frac{pN+2}{N+2}}
+  \|u_0\|_{p}^{p} 
 +C'T,\end{equation}
where $\hat{C}=\tilde{C}-\epsilon>0$.
	
Thus, provided that $ \|f\|_{\frac{pN+2}{N+2},\frac{pN+2}{N+2}}$ is bounded, i.e. $f\in L^{\frac{pN+2}{N+2},\frac{pN+2}{N+2}}$ and $u_0\in L^p$, 
the estimate \eqref{h2myy} implies  
$$
f\in L^{\frac{pN+2}{N+2},\frac{pN+2}{N+2}} \quad \text{and}\quad u_0\in L^p \quad \Rightarrow \quad
u \in L^{p,\infty}\cap L^{q,p},
$$
where $q=\frac{pN}{N-2}$ and the involved constants grow at most polynomially in time. 
\end{proof}

The following Lemma \ref{unequal} will allow to perform a bootstrap argument build on estimate   \eqref{apri1} of 
Lemma \ref{equal} in order to improve the integrability of $f$.

We point out that estimate   \eqref{apri1} leads naturally to a-priori estimates for $u$ which exhibit larger time- than space-integrability (see e.g. 
Lemma \ref{LLInterpol} for the interpolation of the space $L^{p,\infty}$ and  $L^{q,p}$).
Therefore, the following Lemma \ref{unequal} provides a suitable 
variant of Lemma \ref{equal} 
based on the assumption that $f$ features higher time- than space-integrability. More precisely, we shall assume $f\in L^{r,p}$ where $r=r(p,N)$ is a specific space-integrability exponent with $r(p,N)<p$.

\begin{Lem}[$\mu<\gamma$ estimate]\label{unequal}
Consider $N\geq3$. Assume for some $p>1$ that $u_0\in L^p$ and
$f\in L^{r,p}$ with 
$$
r=r(p,N):=\frac{pN}{N+2(p-1)}= p - \frac{2p(p-1)}{N+2(p-1)}<p.
$$ 
Then,
$$
u \in  L^{p,\infty}\cap L^{q,p},
$$
where as above $q=q(p,N) = \frac{pN}{N-2}>p$.

Moreover, by using interpolation in Bochner spaces (see Lemma \ref{LLInterpol}), we have for any interpolation exponent
$\theta\in(0,1)$: 
\begin{equation}\label{apri2}
u\in L^{p,\infty}\cap L^{q,p}\hookrightarrow L^{\sigma,\tau} \ \Rightarrow\ f\in L^{\frac{\sigma}{2},\frac{\tau}{2}}
\qquad
\text{with}
\qquad \sigma(\theta,p)=\frac{pN}{N+2\theta-2},\quad \tau(\theta,p)=\frac{p}{1-\theta}.
\end{equation}
\end{Lem}
\begin{proof}
We return to the proof of Lemma \ref{Lemma31} and consider $\mu<\gamma$ in order to derive a-priori estimates on $u$ based on $f$ featuring higher time- rather than space-integrability.
In particular, we can set $\mu'(p-1)=q$ in \eqref{h1}, having in mind that the $L^q$ norm is the highest possible space, which can be controlled by the the left hand side of \eqref{h1}, see also Remark \ref{case2}. Thus, we get:
\begin{equation}\label{h11}
 \frac{d}{dt}  \|u\|_{p}^{p} 
 + \tilde{C} \|u\|_{q}^{p}
 \le 
 p\|f\|_{r}\|u\|_{q}^{p-1}
 +C'
\end{equation}
where the index 
$$
r=r(p,N):=\frac{q}{1+q-p}=\frac{pN}{N-2+2p}=p - \frac{2p(p-1)}{N+2(p-1)}< p, \qquad \forall \ p>1, \ N\ge3,
$$ 
denotes the minimal space-integrability requirement on $f$ such that the corresponding 
term $\|u\|_{q}^{p-1}$ can still be controlled by the left hand side term $\|u\|_{q}^{p}$.
Next, we integrate \eqref{h11} over $(0,T)$ for any $T>0$
$$
  \|u(T)\|_{p}^{p} 
 +   \tilde{C}\int_0^T\|u\|^p_q\,dt
 \le 
  \|u_0\|_{p}^{p} 
 +p\int_0^T\|f\|_{r}\|u\|_{q}^{p-1}\,dt
 +C'T
$$
and apply on the right hand side H\"older's inequality with $1=\frac{1}{\gamma'}+\frac1\gamma$ to obtain
\begin{equation}\label{h22}
  \|u(T)\|_{p}^{p} 
 +   \tilde{C} \|u\|_{q,p}^p
  \le 
  \|u_0\|_{p}^{p} 
 +p\|f\|_{r,\gamma}
\|u\|_{q,\gamma'(p-1)}^{p-1}
 +C'T.
 \end{equation}
Thus, we set $\gamma'(p-1)=p$ and $\gamma=p$ in \eqref{h22}
\begin{equation*}
  \|u(T)\|_{p}^{p} 
 +   \tilde{C} \|u\|_{q,p}^p
  \le 
  \|u_0\|_{p}^{p} 
 +p\|f\|_{r,p}
{\|u\|_{q,p}^{p-1}}
 +C'T
 \end{equation*}
and apply once more Young's inequality to the right hand side term ${\|u\|_{q,p}^{p-1}}$
such that for a sufficiently small $\eps>0$ 
\begin{equation}\label{h3a}
  \|u(T)\|_{p}^{p}  +  \hat{C} \|u\|_{q,p}^p \le 
  \|u_0\|_{p}^{p} 
 +C(\eps,p)\|f\|_{r,p}^{p} 
 +C'T,
\end{equation}
where $\hat{C} = \tilde{C}-\eps>0$.

Altogether, if $ \|f\|_{r,p}$ is bounded, i.e. $f\in L^{r,p}$ and $u_0\in L^p$, we have
$$
f\in L^{r,p} \quad \text{and}\quad u_0\in L^p \quad \Rightarrow \quad
u \in  L^{p,\infty}\cap L^{q,p}.
$$
We remark that in $f\in L^{r,p}$, since we have set $\gamma(p-1)=p$ and $\gamma'=p$ in \eqref{h22}, the time-integrability index $p$ represents the minimal 
required time-regularity such that estimate \eqref{h3a} holds.

In the last step of the proof of Lemma \ref{unequal}, we use the Bochner space interpolation  Lemma \eqref{LLInterpol} (see Appendix) and get 
$$ 
L^{p,\infty}\cap L^{q,p}\hookrightarrow  L^{\sigma,\tau},
$$
where $\sigma$ and $\tau$ are defined depending on an interpolation parameter $\theta\in(0,1)$ by 
$$
\frac1\sigma=\frac{\theta}{p}+\frac{1-\theta}{q}\qquad
\text{and}\qquad
\frac1\tau=\frac{\theta}{\infty}+\frac{1-\theta}{p},
$$
which leads to 
$$
\sigma(\theta,p)=\frac{pN}{N+2\theta-2},\qquad \tau(\theta,p)=\frac{p}{1-\theta}.
$$
\end{proof}

Finally, the following Lemma \ref{weakcomparisonbootstrap} provides a bootstrap argument based on ideas of
\cite{KarSuz,LSY,SuzYam}, which combines semigroup estimates with weak comparison arguments and allows to bootstrap polynomially in time growing $L^\infty$-bounds.

\begin{Lem}[Weak comparison bootstrap argument]\label{weakcomparisonbootstrap} \hfill\\
Consider $\nu\ge2$ and let $u_i(t,\cdot)$ be weak solutions to the systems \eqref{sys4} or \eqref{sys4GEN}
for $i=1,\ldots,k$. In particular we have $\nu=2$ and $k=4$ for solutions $u_i(t,\cdot)$ of system \eqref{sys4}. 
Assume that for an exponent 
\begin{equation}\label{bootstrapcon}
\mu_0>\max\Bigl\{\nu,\frac{(\nu-1)N}{2}\Bigr\}
\end{equation}
and for all $i=1,\ldots,k$ holds
$$
\|u_i(t,\cdot)\|_{L^{\mu_0}(\Omega)} \le C(t),\qquad \text{for all}\quad t\ge 0,
$$
where $C(t)>0$ is a continuous function in time, which grows polynomially. 

Then, the solution $u_i(\cdot,t)$ satisfies 
\begin{equation} 
\Vert u_1(t,\cdot)\Vert_\infty+\Vert u_2(t,\cdot)\Vert_\infty+\Vert u_3(t,\cdot)\Vert_\infty+ \Vert u_4(t,\cdot)\Vert_\infty\leq C(t), \qquad\text{for all}\quad t> 0, 
 \label{eqn:ue}
\end{equation} 
where $C(t)>0$ is again a (different) continuous function in time on $(0,\infty)$, which grows polynomially. 
 \label{l1}
\end{Lem}

\begin{proof}
We remind that each $(u_i)_{ i=1\ldots k}$ satisfies
\begin{equation*}
    \left\{
      \begin{aligned}
        &\partial_t u_i - d_i \,\Delta_x u_i = f_i, &&\quad \text{ on } \quad \Omega_T,
        \\
        &\nabla_x u_i \cdot n(x) = 0, &&\quad \text{ on } \quad  [0,T]
        \times \partial \Omega,
        \\
        &u_i(0,x) = u_{i0}(x)\ge0, &&\quad \text{ for } \quad x \in \Omega.
              \end{aligned}
    \right.
  \end{equation*}
In particular, the concentrations $(u_i)_{ i=1\ldots k}$ are non-negative and the considered nonlinearities $|f_i|\le C \sum_{j=1}^k u_j^{\nu}$ are super-quadratic in the $u_i$ of order $\nu\ge2$. 
 This motivates us, by neglecting the negative terms of $f$, to consider: 
 $$
\partial_t u_i\leq d_i\Delta u_i+Cu^{\nu}, \qquad i=1,\ldots,k
 $$
where $u^\nu=\sum_{j=1}^k u_j^{\nu}$ for simplicity 
and $d_i$ is the diffusion rate of $u_i$.

Now, we define $g_i=\lambda u_i+Cu^\nu$ for $\lambda>0$ and all $i=1,\ldots,k$ 
and consider the linear comparison problem:
\begin{equation}
\label{p1}
\begin{cases}
\overline{u}_t=d_i\Delta \overline{u}-\lambda \overline{u}+g_i(t,\cdot),&\quad \text{ on } \quad \Omega_T,\\
\nabla_x \overline{u} \cdot n(x) = 0,&\quad \text{ on } \quad  [0,T]
        \times \partial \Omega,\\ 
\overline{u}_0=u_{i0}(x)\ge0,&\quad \text{ for } \quad x \in \Omega,
\end{cases}
\end{equation}
where $g_i(t,\cdot)$ is considered a given function satisfying $\|g(t,\cdot)\|_{L^{\frac{{\mu_0}}{\nu}}(\Omega)}<C_1(t)$
since $\|u(t,\cdot)\|_{L^{\mu_0}(\Omega)}<C(t)$ and $C_1(t)$ has the same properties as $C(t)$.

Thus, by the comparison principle of weak solutions of parabolic equations with values in the reflexive Banach space $L^{\frac{{\mu_0}}{\nu}}(\Omega)$ for $\frac{{\mu_0}}{\nu}>1$ (see e.g. \cite[Theorem 11.9]{Chipot}), we have that $0\leq u\leq \overline{u}$.

Next, we refer e.g. to Rothe \cite{rothe} for the following semi-group estimates of the Laplace operator 
subject to Neumann boundary conditions 
\[ 
\Vert (e^{t\Delta}u)(t,\cdot)\Vert_r\leq C(q,r)\max\{1, t^{-\frac{N}{2}(\frac{1}{q}-\frac{1}{r})}\}\Vert u(t,\cdot)\Vert_q, \qquad 1\leq q\leq r\leq \infty, 
\] 
with $C(q,r)>0$, which implies for  $L_i=d_i\Delta-\lambda$ subject to Neumann boundary condition
\begin{equation*} 
\Vert (e^{tL_i}u)(t,\cdot)\Vert_r\leq C(q,r)e^{-\lambda t}\max\{ 1,(d_it)^{-\frac{N}{2}(\frac{1}{q}-\frac{1}{r})}\}\Vert u(t,\cdot)\Vert_q, \qquad 1\leq q\leq r\leq \infty. 
\end{equation*}
Equation \eqref{p1} implies for all $\tau>0$
\begin{equation} 
\overline{u}(t,\cdot)=e^{(t-\tau)L_i}u_{i}(\tau)+\int_\tau^te^{(t-s)L_i}g(s,\cdot)\,ds. 
 \label{eqn:int-0}
\end{equation} 
By taking the $r$-norm of (\ref{eqn:int-0}) and applying the above semi-group estimate with $q=\frac{{\mu_0}}{\nu}$, we obtain
\begin{equation}\label{varconst2}
\|\overline{u}(t,\cdot)\|_r\leq \|e^{(t-\tau)L_i}u_{i}(\tau)\|_r +  \int_\tau^t e^{-\lambda(t-s)}
\max\{ 1,(d_i\,(t-s))^{-\frac{N}{2}(\frac{\nu}{{\mu_0}}-\frac{1}{r})}\}\,C_1(s)\,ds
\end{equation}
and at this point, in order to have the above right hand side integrable near the singularity $s=t$, we need
$$
-\frac{N}{2}\left(\frac{\nu}{{\mu_0}}-\frac1r\right)>-1\qquad\Rightarrow\qquad r<\frac{N{\mu_0}}{N\nu-2{\mu_0}}\quad\text{provided that}\quad \frac{{\mu_0}}{\nu}<\frac{N}{2}, 
$$
while for $\frac{{\mu_0}}{\nu}=\frac{N}{2}$, we can chose any $r<\infty$ and for  $\frac{{\mu_0}}{\nu}>\frac{N}{2}$, we can set $r=\infty$.
In all these cases, we can estimate eq. \eqref{varconst2} as
\begin{equation*}
\|\overline{u}(t,\cdot)\|_r\leq C(\tau) + \sup_{\tau\le s \le t} C_1(s)
\int_\tau^t 
\left(1+(d_i(t-s))^{-\frac{N}{2}(\frac{\nu}{{\mu_0}}-\frac{1}{r})}\right)ds
\leq C(\tau) + \sup_{\tau\le s \le t} C_1(s) \,O(t) 
\end{equation*}
and the $r$-norm is bounded by a polynomial of order $n+1$, if $C_1$ is a polynomial of order $n$.
Recalling that $0\leq u\leq \overline{u}$, we conclude that 
\begin{equation}\label{varconst4}
\|{u}(t,\cdot)\|_r\leq C_2(t),\qquad t\ge\tau>0,
\end{equation}
where $C_2(t)$ grows polynomially in time. 

In order to bootstrap the estimates  \eqref{eqn:int-0}--\eqref{varconst4}, we require $\mu_1:=r>{\mu_0}>\nu$, which is possible  
provided that 
$$
{\mu_0} < \frac{N{\mu_0}}{N\nu-2{\mu_0}} \quad\iff\quad {\mu_0} > \frac{(\nu-1)N}{2}.
$$
Thus, we find that assumption \eqref{bootstrapcon} allows to start a bootstrap argument, which leads after finitely many steps 
to  
\begin{equation*}
\|{u}(t,\cdot)\|_{\mu_n}\leq C_{n+1}(t),\qquad t\ge\tau>0,
\end{equation*}
where $\frac{\mu_n}{\nu}\ge \frac{N}{2}$ and $C_{n+1}(t)$ grows polynomially in time.
Thus, after one or two (if $\frac{\mu_n}{\nu}= \frac{N}{2}$) further bootstrap steps, 
we reach  
\[ 
\Vert u(\cdot,t)\Vert_\infty \leq C(t), \qquad t\ge\tau> 0,
\] 
where for all $\tau>0$ the bound $C(t)$ grows polynomially in time.
\end{proof}

We can now present the proof of our main Theorem.

\subsubsection*{Proof of Theorem \ref{Theorem} (Global existence of classical solutions to system \eqref{sys4})}

The first part of the proof establishes a bootstrap argument based on the Lemmata \ref{equal} and \ref{unequal} in order to prove the following gain of regularity: Provided that $\Gamma> \frac{N+2}{2}-\frac{2N}{N+2}$ and initial data $u_{i0}\in L^{\frac{3N}{4}}(\Omega)$, then
\begin{equation*}
\label{bootstrap}u\in L^{2\Gamma,2\Gamma}\quad\Rightarrow\quad f\in L^{\Gamma,\Gamma} 
\quad\Rightarrow\quad u\in L^{\sigma_0,\tau_0} 
\quad\Rightarrow\quad u\in L^{\sigma_1,\tau_1}
\quad\Rightarrow\quad
\ldots \quad\Rightarrow\quad u\in L^{\frac{3N}{4},\infty}.
\end{equation*}
Once we know that $u\in L^{\frac{3N}{4},\infty}$, then Lemma \ref{weakcomparisonbootstrap} implies $u\in L^{\infty,\infty}$
with constants, which grow at most polynomially in time. In a final step, we shall then interpolate polynomially growing a-priori estimates with the exponentially convergence of solutions towards equilibrium in order to prove uniform-in-time boundedness of the solutions.

\begin{flushleft}
{\it Initial Step: A-priori estimate via duality Proposition \ref{propone}}
\end{flushleft}
 
In order to start our argument, we recall assumption \eqref{deltacon1} and estimate \eqref{uapriori} of Proposition \ref{propone}, which 
ensures that $u\in L^{2\Gamma,2\Gamma}$ for $\Gamma > \frac{N+2}{2}-\frac{2N}{N+2}>1$ for $N\ge3$. 
We remark again that for $\Gamma\ge\frac{N+2}{2}$, the statement of the Theorem \ref{Theorem}
follows from well known parabolic regularity results, which directly ensure $u\in L^{\infty}(\Omega_T)$ for all $T>0$,
see e.g. \cite{CDF}, which also proves the polynomial growth of the constants in $T$.
\medskip

Given the a-priori estimate $u\in L^{2\Gamma,2\Gamma}$, we apply estimate \eqref{apri1} of Lemma \ref{unequal}, 
$$
u\in L^{2\Gamma,2\Gamma}
\quad\Rightarrow \quad
f\in L^{\Gamma,\Gamma}\quad \text{and}\quad u_0\in L^{p_0}\quad\stackrel{\eqref{apri1}}{\Rightarrow}\quad 
u\in L^{p_0,\infty}\cap L^{q_0,p_0},
$$
where we have set 
\begin{equation*}
\Gamma=\frac{p_0N+2}{N+2}
\quad\Rightarrow\quad p_0=\frac{\Gamma(N+2)-2}{N}> \frac{N}{2}\qquad \text{for} \qquad \Gamma > \frac{N+2}{2}-\frac{2N}{N+2}.
\end{equation*}
Now, if we have $p_0>\frac{3N}{4}$ and, thus, $u\in L^{\frac{3N}{4},\infty}$, then Lemma \ref{weakcomparisonbootstrap} directly implies  $u\in L^{\infty,\infty}$ since condition \eqref{bootstrapcon} is satisfied with $p_0=\mu_0>\frac{3N}{4}>\max\{2,\frac{N}{2}\}$ for $N\ge3$. 
\medskip

In the following, it remains to consider the cases $p_0<\frac{3N}{4}$.
By applying the interpolation Lemma \ref{LLInterpol}, we get for all $\theta_0\in(0,1)$ (and 
$\frac{1}{\sigma_0}=\frac{\theta_0}{p_0}+\frac{1-\theta_0}{q_0}$ and $\frac{1}{\tau_0}=\frac{\theta_0}{\infty}+\frac{1-\theta_0}{p_0}$)
$$
u\in L^{p_0,\infty}\cap L^{q_0,p_0}\hookrightarrow  L^{\sigma_0,\tau_0}\qquad 
\text{with}\qquad
\sigma_0(\theta_0,p_0)=
\frac{p_0N}{N+2\theta_0-2},
\quad 
\tau_0(\theta_0,p)=
\frac{p_0}{1-\theta_0},
$$
which gives the new $f-$regularity:
$$
f\in L^{\frac{\sigma_0}{2},\frac{\tau_0}{2}}.
$$
Since we have $\infty>q_0$ in the above interpolation $L^{p_0,\infty}\cap L^{q_0,p_0}$, it follows that the interpolated time-integrability is always larger than the corresponding space-integrability, i.e. $\tau_0>\sigma_0$. 

In the following bootstrap, we shall therefore chose $\theta_0\in(0,1)$ in order to apply Lemma \ref{unequal}.
In fact in view of applying Lemma \ref{unequal}, we need to set 
$$
\frac{\tau_0}{2}=p_1\qquad\text{and}\qquad\frac{\sigma_0}{2}=r_1(p_1,N)=\frac{p_1N}{N+2(p_1-1)},
$$
where the later choice imposes the condition to set $\theta_0$ to the value  
$$
\theta_0=\frac{p_0}{N}\in(0,1)\qquad \text{since}\qquad p_0<\frac{3N}{4}.
$$
Thus, for all $p_0<\frac{3N}{4}$, it is admissible to set $\theta_0=\frac{p_0}{N}$ and we can bootstrap again via Lemma \ref{unequal}
$$
f\in L^{\frac{\sigma_0}{2},\frac{\tau_0}{2}} \quad \text{and}\quad u_0\in L^{p_1}\quad\stackrel{\eqref{apri2}}{\Rightarrow} \quad
u\in L^{p_1,\infty}\cap L^{q_1,p_1} \quad\Rightarrow\quad f\in L^{\frac{\sigma_1}{2},\frac{\tau_1}{2}}
$$
with
\begin{equation*}
p_1 = \frac{\tau_0}{2} = \frac{p_0}{2(1-\theta_0)} = \frac{N p_0 }{2(N-p_0)},
\qquad \sigma_1(\theta_1,p_1)=\frac{p_1N}{N+2\theta_1-2},\qquad \tau_1(\theta_1,p)=\frac{p_1}{1-\theta_1}
\end{equation*}
If now $p_1\ge \frac{3N}{4}$, then $u\in L^{\frac{3N}{4},\infty}$ and Lemma \ref{weakcomparisonbootstrap} implies $u\in L^{\infty,\infty}$.
Otherwise, by repeating the above argument, we choose again 
$\theta_1=\frac{p_1}{N}\in(0,1)$ (since $p_1<\frac{3N}{4}$) 
in order to have
$$
f\in L^{\frac{\sigma_1}{2},\frac{\tau_1}{2}} \quad \text{and}\quad u_0\in L^{p_2}\quad\stackrel{\eqref{apri2}}{\Rightarrow} \quad
u\in L^{p_2,\infty}\cap L^{q_2,p_2} \quad\Rightarrow\quad f\in L^{\frac{\sigma_2}{2},\frac{\tau_2}{2}}
$$
with 
\begin{equation*}
p_2 = \frac{\tau_1}{2} = \frac{p_1}{2(1-\theta_1)} = \frac{N p_1}{2(N-p_1)}.
\end{equation*}

It is now straightforward to verify that the following bootstrap recursion (where we have always chosen $\theta_n= \frac{p_n}{N}$)  
\begin{equation*}
p_{n+1} = \frac{Np_n}{2(N-p_n)}, \qquad p_0>\frac{N}{2},
\end{equation*}
has the unique fixed point $p_{\infty} = \frac{N}{2}$ and is strictly monotone increasing for $p_{n} > \frac{N}{2}$ as long as we can chose 
$\theta_n= \frac{p_n}{N}$, i.e. as long as  $p_n<\frac{3N}{4}$. 
Thus, after finitely many iteration steps, we have $p_{n+1}\ge\frac{3N}{4}$ and thus $u\in L^{\frac{3N}{4},\infty}$ and Lemma \ref{weakcomparisonbootstrap} implies $u\in L^{\infty,\infty}$.

\medskip

In a final step, we shall prove uniform-in-time boundedness of the obtained solutions $u\in L^{\infty}(\Omega_T)$.
First we recall that it is well known that weak solutions $u_i\in L^{\infty}(\Omega_T)$ are in fact classical solutions due to standard parabolic 
regularity results. 

Then, for classical solutions $u_i\in L^{\infty}(\Omega_T)$, for which the $L^{\infty}(\Omega)$-norm grows at most polynomially in time, parabolic regularity arguments show also that $\|u_i(t,\cdot)\|_{H^k(\Omega)}$ for all $k\ge1$ grows at most polynomially in time, 
see e.g. \cite{DF2}. 

Thus, Gagliadro-Nirenberg interpolation between $L^1$ and $H^k$ for $k>\frac{N}{2}$ yields a uniform-in-time $L^{\infty}$ bound, i.e. there exists a $\theta \in(0,1)$ such that for all $T>0$
\begin{align*}
\|u_i(\cdot,T)\|_{L^{\infty}(\Omega)} &\le \|u_i(\cdot,T)-u_{i,\infty}\|_{\infty} + \|u_{i,\infty}\|_{\infty} \nonumber \\
&\le G \|u_i(\cdot,T)-u_{i\infty}\|_1^\theta \|u_i(\cdot,T)-u_{i,\infty}\|_{H^k}^{1-\theta}+\|u_{i,\infty}\|_{\infty}, \nonumber \\
&\le C(G,C_1) e^{-\frac{C_2\theta}{2} T} C_T^{1-\theta}+C(M)<C, 
\end{align*}
where we have used the exponential convergence to equilibrium stated in Proposition \ref{Convergence} and the polynomial-in-time dependence of the 
constant $C_T$ to obtain the uniform-in-time $L^{\infty}$-bound \eqref{Linftybound}.

This concludes the proof of Theorem \ref{Theorem}.
\hfill $\square$

\begin{Rem}
For $N=3$ and $p_0=\frac{N}{2}=\frac{3}{2}$, we calculate $\Gamma=\frac{13}{10}$ and thus observe that 
$$
\sigma_0 =\frac{9}{4} < \frac{13}{5} = 2\Gamma, \qquad\text{and}\qquad
\tau_0 = 3 > \frac{13}{5} = 2\Gamma.
$$
In fact, we expect in general that $\sigma_0<2\Gamma<\tau_0$.
\end{Rem}

\subsection{Proof Corollary \ref{Corollary}}
\begin{proof}
First, we remark that thanks due the conservation laws assumed in Corollary \ref{Corollary}, we observe that for every 
$i=1,\ldots,k$ the sum $z_i:= \sum_{j=1}^k a^i_j u_j$ where $a^i_i>0$ satisfies a problem of the form 
\begin{equation*}
\begin{cases}
\partial_t z_i - \Delta \left(A_i\,z_i\right) \le 0, \\
 n(x)\cd\nx z_i(t,x) = 0, \\
A_i(t,x):= \frac{\sum_{j=1}^k d_i \,a^i_j\,u_j}{\sum_{i=1}^k a^i_j\,u_j} \in [a,b],
\end{cases}
\end{equation*} 
where $0<a:=\inf_{i= 1..k} \{d_i\}$ and $b:=\sup_{i= 1..k} \{d_i\}$.

Thus, the $\delta$-smallness condition \eqref{deltacon2} and Proposition \ref{propone} ensure for all $T>0$ 
the formal a-priori estimate $u_i\in L^{\nu \Gamma}(\Omega_T)$ and thus for all $i=1,\ldots,k$, $f_i(u)\in L^{\Gamma}$ and it is easy verified that $\Gamma>1$. Thus, the existence of weak global solutions to 
system \eqref{sys4GEN} can be shown as the limit of global solutions to suitable approximating systems: 
By using, for instance, approximating systems, which truncate the nonlinearities $f_i(u)$ following the lines of e.g. Michel Pierre \cite{Pie03}, we can 
pass to the limit in system \eqref{sys4GEN} and in particular in the nonlinear terms on the right hand side $f_i(u)\in L^{\Gamma}$ for $\Gamma>1$. This will yield global weak solutions of \eqref{sys4GEN} satisfying $f_i(u)\in L^{\Gamma}$.
\medskip

In the following, we shall prove that these global weak solutions are indeed classical solutions. 
This is done analog to the proof of Theorem  \ref{Theorem}:  	
We start by proving that $u\in L^{\nu\Gamma,\nu\Gamma}$ with initial data $u_{i0}\in L^{\frac{(2\nu-1)N}{4}}(\Omega)$ can be bootstrapped to $u\in L^{\frac{(2\nu-1)N}{4},\infty}$
provided that
$$
\Gamma>\frac{(\nu-1)N^2+4}{2 (N+2)}.
$$
Given the a-priori estimate $u\in L^{\nu\Gamma,\nu\Gamma}$, we apply estimate \eqref{apri1} of Lemma \ref{unequal}, 
$$
u\in L^{\nu\Gamma,\nu\Gamma}
\quad\Rightarrow \quad
f\in L^{\Gamma,\Gamma} \quad \text{and}\quad u_0\in L^{p_0}\quad\stackrel{\eqref{apri1}}{\Rightarrow}\quad 
u\in L^{p_0,\infty}\cap L^{q_0,p_0},
$$
where we have set 
\begin{equation*}
\Gamma=\frac{p_0N+2}{N+2}
\quad\Rightarrow\quad p_0=\frac{\Gamma(N+2)-2}{N}> \frac{(\nu-1)N}{2}\qquad \text{for} \qquad \Gamma > \frac{(\nu-1)N^2+4}{2(N+2)}.
\end{equation*}
Thus, in case that $p_0\ge \frac{(2\nu-1)N}{4}$, we have $u\in L^{\frac{(2\nu-1)N}{4},\infty}$ and Lemma \ref{weakcomparisonbootstrap} implies  $u\in L^{\infty,\infty}$ since $p_0=\mu_0 \ge \frac{(2\nu-1)N}{4} >\min\{\nu,\frac{(\nu-1)N}{2}\}$ for $N\ge 3$.
\medskip

It remains to consider $p_0<\frac{N(2\nu-1)}{4}$.
By applying the interpolation Lemma \ref{LLInterpol}, we get for all $\theta_0\in(0,1)$ (and 
$\frac{1}{\sigma_0}=\frac{\theta_0}{p_0}+\frac{1-\theta_0}{q_0}$ and $\frac{1}{\tau_0}=\frac{\theta_0}{\infty}+\frac{1-\theta_0}{p_0}$)
$$
u\in L^{p_0,\infty}\cap L^{q_0,p_0}\hookrightarrow  L^{\sigma_0,\tau_0}\qquad 
\text{with}\qquad
\sigma_0(\theta_0,p_0)=
\frac{p_0N}{N+2\theta_0-2},
\quad 
\tau_0(\theta_0,p)=
\frac{p_0}{1-\theta_0},
$$
which gives the new $f-$regularity:
$$
f\in L^{\frac{\sigma_0}{\nu},\frac{\tau_0}{\nu}}.
$$
We proceed analog to Theorem \ref{Theorem} and therefore chose $\theta_0\in(0,1)$ by setting 
$$
\frac{\tau_0}{\nu}=p_1\qquad\text{and}\qquad\frac{\sigma_0}{\nu}=r_1(p_1,N)=\frac{p_1N}{N+2(p_1-1)},
$$
where the later choice imposes the condition to set $\theta_0$ to the value  
$$
\theta_0=\frac{2p_0}{\nu N}\in(0,1), \qquad \text{since} \quad p_0<\frac{(2\nu-1)N}{4}.
$$

Thus, by applying a bootstrap recursion similar to Theorem \ref{Theorem}, we have 
 (by setting always $\theta_n= \frac{2p_n}{\nu N}$)  
\begin{equation}\label{recus}
p_{n+1} = \frac{\tau_n}{\nu} = \frac{p_n}{\nu(1-\theta_n)} = \frac{Np_n}{\nu N-2p_n}, \qquad p_0>\frac{(\nu-1)N}{2}.
\end{equation}
It is then straightforward to verify that the recursion \eqref{recus} has a unique fixed point $p_{\infty} = \frac{(\nu-1)N}{2}$ and is monotone increasing for $p_{n} > \frac{(\nu-1)N}{2}$  as long as we can chose 
$\theta_n= \frac{p_n}{N}$, i.e. as long as  $p_n<\frac{(2\nu-1)N}{4}$. 

Thus, after finitely many iteration steps, we obtain $p_n\ge\frac{(2\nu-1)N}{4}$ and thus $u\in L^{\frac{(2\nu-1)N}{4},\infty}$ and Lemma \ref{weakcomparisonbootstrap} implies $u\in L^{\infty,\infty}$. Thus, by applying standard parabolic argument, we obtain classical solutions and the statement of Corollary \ref{Corollary} is proven.  
\end{proof}

\section{Appendix}\label{app}
We remark that the following result is certainly not new but since we couldn't find any reference we provide a proof for the convenience of the reader.
	
\begin{Lem}[Interpolation in Bochner spaces] \label{LLInterpol}
Consider $p_1,p_2\in[1,+\infty]$ with $p_1 \neq p_2$ and $q_1,q_2\in[1,+\infty]$ with $q_1 \neq q_2$.

Then,  
$$
L^{p_1,q_1}\cap L^{p_2,q_2}\hookrightarrow  L^{s,r}
$$
and we have  for any $\theta\in(0,1)$
$$
s\in(p_1,p_2) \quad \text{with}\quad \frac{1}{s}=\frac{\theta}{p_1}+\frac{1-\theta}{p_2}, \qquad\text{and}\qquad
r\in(q_1,q_2) \quad \text{with}\quad \frac{1}{r}=\frac{\theta}{q_1}+\frac{1-\theta}{q_2}.
$$
\end{Lem}

\begin{proof}

	We start with the spatial $L^s$-norm:
	$$
	\int_\Omega |u|^s\,dx=
	\int_\Omega |u|^{s\theta}|u|^{s(1-\theta)}\,dx
	\leq
	\left(\int_\Omega|u|^{s\theta p}\,dx\right)^{\frac1p}
	\left(\int_\Omega|u|^{s(1-\theta) p'}\,dx\right)^{\frac{1}{p'}}
	$$
	for any $\theta\in(0,1)$ after applying H\"older's inequality with $\frac1p+\frac{1}{p'}=1$. Next, we integrate over $(0,T)$ after raising to the power $\frac{r}{s}$:
		$$
	\int_0^T\left(\int_\Omega |u|^s\,dx\right)^{\frac{r}{s}}dt
	\leq
	\int_0^T\left(\int_\Omega|u|^{s\theta p}\,dx\right)^{\frac{r}{ps}}
	\left(\int_\Omega|u|^{s(1-\theta) p'}\,dx\right)^{\frac{r}{p's}}dt.
	$$
	By setting $\frac{1}{s p} = \frac{\theta}{p_1}$ and $\frac{1}{s p'} = \frac{1-\theta}{p_2}$ (from which follows $\frac{1}{s}=\frac{\theta}{p_1}+\frac{1-\theta}{p_2}$ since $\frac1p+\frac{1}{p'}=1$ and thus $s\in(p_1,p_2)$ for $\theta\in(0,1)$), we have
			$$
	\int_0^T\left(\int_\Omega |u|^s\,dx\right)^{\frac{r}{s}}dt
	\leq
	\int_0^T\left(\int_\Omega|u|^{p_1}\,dx\right)^{\frac{r\theta}{p_1}}
	\left(\int_\Omega|u|^{p_2}\,dx\right)^{\frac{r(1-\theta)}{p_2}}dt.
	$$
	Applying H\"older's inequality (in time) with  $1=\frac{1}{\mu}+\frac{1}{\mu'}$ yields
	$$
	\int_0^T\left(\int_\Omega |u|^s\,dx\right)^{\frac{r}{s}}dt
	\leq
		\left(\int_0^T\left(\int_\Omega|u|^{p_1}\,dx\right)^{\mu\frac{r\theta}{p_1}}dt\right)^\frac{1}{\mu}
	\left(\int_0^T\left(\int_\Omega|u|^{p_2}\,dx\right)^{\mu'\frac{r(1-\theta)}{p_2}}dt\right)^\frac{1}{\mu'}.
	$$
	By setting again $\frac{1}{r \mu} = \frac{\theta}{q_1}$ and $\frac{1}{r \mu'} = \frac{1-\theta}{q_2}$ (from which follows again $\frac{1}{r}=\frac{\theta}{q_1}+\frac{1-\theta}{q_2}$ since $\frac1\mu+\frac{1}{\mu'}=1$ and thus $r\in(q_1,q_2)$ for $\theta\in(0,1)$), we derive after raising to the power $\frac1r$:
			$$
	\|u\|_{s,r}\equiv\left(\int_0^T\|u\|_s^rdt\right)^\frac1r
	\leq
	\left(\int_0^T\|u\|_{p_1}^{q_1}dt\right)^\frac{1}{r\mu}
	\left(\int_0^T\|u\|_{p_2}^{q_2}dt\right)^\frac{1}{r\mu'}=\|u\|_{p_1,q_1}^{\theta}\|u\|_{p_2,q_2}^{1-\theta}
	$$
\end{proof}

\section{Acknowledgements }
This work has partially been supported by NAWI Graz. K.F. and E.L. gratefully acknowledge the kind hospitality of the University of Osaka.


\begin{thebibliography}{10}

\bibitem{Ama} H. Amann, \textit{Global existence for semilinear parabolic problems.} 
J. Reine Angew. Math. \textbf{360}, (1985), pp. 47--83.

 
\bibitem{BD} M. Bisi, L. Desvillettes,  \textit{From reactive Boltzmann 
equations to reaction-diffusion systems.} J. Stat. Phys. \textbf{125} (2006), 
no. 1, pp. 249--280.

\bibitem{BCD} M. Bisi, F. Conforto, L. Desvillettes, 
\textit{Quasi-steady-state approximation for reaction-diffusion equations,} Bull. 
Inst. Math., Acad. Sin. (N.S.), \textbf{2} (2007), pp. 823-850.

\bibitem{BP} D. Bothe, M. Pierre, 
\textit{Quasi-steady-state approximation for a reaction-diffusion system with fast 
intermediate,} J. Math. Anal. Appl. \textbf{368}, n.1 (2010) pp. 120-132.

 
\bibitem{CDF} J. A. Ca{\~n}izo, L. Desvillettes, K. Fellner,  \textit{Improved duality estimates and applications to reaction-diffusion equations,} Comm. Partial Differential Equations, \textbf{39} no.6 (2014) 1185--1204.

\bibitem{CV} C. Caputo, A. Vasseur, \textit{Global regularity of solutions to systems of reaction-diffusion with sub-quadratic growth in any dimension}, Commun. Partial Differential Equations \textbf{34} no.10--12 (2009) pp. 1228--1250. 

\bibitem{Chipot} M. Chipot, Elements of nonlinear analysis, Birkh$\ddot{\mathrm{a}}$user (2000).

 \bibitem{DMP} A. De Masi, E. Presutti,
 \newblock {\em Mathematical methods for hydrodynamic limits,}
 \newblock Springer--Verlag, Berlin, 1991.

\bibitem{DF} L. Desvillettes, K. Fellner,  \textit{Exponential decay toward equilibrium via entropy methods for reaction-diffusion equations,}  J. Math. Anal. Appl. {\bf 319} n.1 (2006) {157--176}.

\bibitem{DF2} L. Desvillettes, K. Fellner,  \textit{Entropy methods for reaction-diffusion 
equations: slowly growing a priori bounds,} Revista Matem\'{a}tica Iberoamericana, \textbf{24} no. 2 (2008) pp. 407--431.

\bibitem{DF3} L. Desvillettes, K. Fellner,  \textit{Exponential Convergence to Equilibrium for a Nonlinear Reaction-Diffusion Systems Arising in Reversible Chemistry} System Modelling and Optimization, IFIP AICT, \textbf{443} (2014) 96--104.

 \bibitem{DFPV} L. Desvillettes, K. Fellner, M. Pierre and J. Vovelle,
 \textit{About global existence for quadratic systems of reaction-diffusion,}
 J. Advanced Nonlinear Studies \textbf{7} no 3. (2007) pp. 491--11. 


\bibitem{GV} T. Goudon, A. Vasseur, \textit{Regularity analysis for systems of 
reaction-diffusion equations,} Ann. Sci. Ec. Norm. Super., (4) \textbf{43} no. 1 (2010) pp. 117--141. 

\bibitem{HM} S.L. Hollis, J.J. Morgan, \textit{On the blow-up of solution to some semilinear and quasilinear reaction-diffusion systems,} Rocky Mountain Journal of Mathematics, \textbf{24} no. 4 (1994) pp. 1447--1465. 

\bibitem{KarSuz} G. Karali and T. Suzuki, \textit{Global-in-time behavior of the solution to a Gierer-Meinhardt system}, Discrete and Continuous Dynamical Systems, \textbf{vol. 33}, 7, p.2885-2900, 2013.

\bibitem{LSY} E. Latos, T. Suzuki, Y. Yamada, \textit{Transient and asymptotic dynamics of a prey-predator system with diffusion}, Math. Meth. Appl. Sci. \textbf{35} (2012) pp. 1101--1109.

\bibitem{Pie03} M. Pierre, \textit{Weak solutions and supersolutions in $L\sp 1$ for reaction-diffusion systems}, J. Evol. Equ., \textbf{3}, no. 1 (2003) 153--168. 

 \bibitem{P2}{M. Pierre, \textit{Global Existence in Reaction-Diffusion Systems with Dissipation of Mass: a Survey}, Milan J. Math.}  \textbf{78}  (2010), no. 2, 417--455.
 
\bibitem{Pdual} M. Pierre, D. Schmitt, \textit{Blowup in reaction-diffusion systems with dissipation of mass,} SIAM Review, \textbf{42} (2000) no. 1, pp. 93--106.

\bibitem{rothe}F. Rothe, Global Solutions of Reaction-Diffusion Equations, {\it Lecture Notes in Mathematics, Springer-Verlag,(1984)}. 

 \bibitem{SuzYam} T. Suzuki, Y. Yamada, \textit{Global-in-time behavior of Lotka-Volterra system with diffusion}, Indiana Univ. Math. \textbf{J. 64 No. 1} (2015), 181–216.


\end{thebibliography}
\end{document}